\documentclass[openany, amssymb, psamsfonts]{amsart}
\usepackage{asymptote}
\usepackage{amssymb}
\usepackage{mathtools}
\usepackage{mathrsfs,comment}
\usepackage[usenames,dvipsnames]{color}
\usepackage[normalem]{ulem}
\usepackage{url}
\usepackage[all,arc,2cell]{xy}
\usepackage{tikz-network}
\usepackage{tikz}
\usetikzlibrary{quotes}
\usetikzlibrary {graphs}
\usetikzlibrary{quotes}
\usepackage{lipsum}
\usepackage{wrapfig}
\usepackage{hyperref}

\begin{asydef}
size(3.5cm);
pen sh = paleyellow;
pair A = dir(150), B = dir(30), C = dir(270), X = -2*A, Y = -2*B, Z = -2*C;
real r = abs(A-B);
pen venn1 = grey+fontsize(8pt);
pen venn2 = fontsize(9pt);
void makeVenn(string s) {
  draw(circle(A,r), black+1.5); draw(circle(B,r), black+1.5); draw(circle(C,r), black+1.5);
  label(rotate(60)*"$N(u)$", (1+3**0.5)*A, A, venn1);
  label(rotate(-39)*"$N(v)$", (1+3**0.5)*B, B, venn1);
  label("$N(w)$", (1+3**0.5)*C, C, venn1);
  label("\textbf{" + s + "}", (0,-4), fontsize(10pt)+blue);
}
void fillRu(string s) { fill(arc(C,A,Y)--arc(A,Y,Z,CW)--arc(B,Z,A)--cycle, sh); label(s, 2*A, venn2); }
void fillRv(string s) { fill(arc(A,B,Z)--arc(B,Z,X,CW)--arc(C,X,B)--cycle, sh); label(s, 2*B, venn2); }
void fillRw(string s) { fill(arc(B,C,X)--arc(C,X,Y,CW)--arc(A,Y,C)--cycle, sh); label(s, 2*C, venn2); }
void fillRuv(string s) { fill(arc(C,A,B,CW)--arc(A,B,Z)--arc(B,Z,A)--cycle, sh); label(s, -1.3*C, venn2); }
void fillRvw(string s) { fill(arc(A,B,C,CW)--arc(B,C,X)--arc(C,X,B)--cycle, sh); label(s, -1.3*A, venn2); }
void fillRuw(string s) { fill(arc(B,C,A,CW)--arc(C,A,Y)--arc(A,Y,C)--cycle, sh); label(s, -1.3*B, venn2); }
void fillRuvw(string s) { fill(arc(A,C,B)--arc(C,B,A)--arc(B,A,C)--cycle, sh); label(s, (0,0), venn2); }
\end{asydef}

\usepackage{graphicx}
\UseAllTwocells
\usepackage{enumerate}

\usepackage{wrapfig}
\hypersetup{%
  bookmarksnumbered=true,%
  bookmarks=true,%
  colorlinks=true,%
  linkcolor=blue,%
  citecolor=blue,%
  filecolor=blue,%
  menucolor=blue,%
  pagecolor=blue,%
  urlcolor=blue,%
  pdfnewwindow=true,%
  pdfstartview=FitBH}

\newcommand{\maincase}[1]{\noindent\textbf{#1}}
\newcommand{\mycomment}[1]{}

\def\makeautorefname#1#2{\expandafter\def\csname#1autorefname\endcsname{#2}}
\def\equationautorefname~#1\null{(#1)\null}
\makeautorefname{footnote}{footnote}%
\makeautorefname{item}{item}%
\makeautorefname{figure}{Figure}%
\makeautorefname{table}{Table}%
\makeautorefname{part}{Part}%
\makeautorefname{appendix}{Appendix}%
\makeautorefname{chapter}{Chapter}%
\makeautorefname{section}{Section}%
\makeautorefname{subsection}{Section}%
\makeautorefname{subsubsection}{Section}%
\makeautorefname{theorem}{Theorem}%
\makeautorefname{thm}{Theorem}%
\makeautorefname{cor}{Corollary}%
\makeautorefname{lem}{Lemma}%
\makeautorefname{prop}{Proposition}%
\makeautorefname{pro}{Property}
\makeautorefname{conj}{Conjecture}%
\makeautorefname{defn}{Definition}%
\makeautorefname{notn}{Notation}
\makeautorefname{notns}{Notations}
\makeautorefname{rem}{Remark}%
\makeautorefname{quest}{Question}%
\makeautorefname{exmp}{Example}%
\makeautorefname{ax}{Axiom}%
\makeautorefname{claim}{Claim}%
\makeautorefname{ass}{Assumption}%
\makeautorefname{asss}{Assumptions}%
\makeautorefname{con}{Construction}%
\makeautorefname{prob}{Problem}%
\makeautorefname{warn}{Warning}%
\makeautorefname{obs}{Observation}%
\makeautorefname{conv}{Convention}%
\makeautorefname{proof}{Proof}%

\newtheorem{thm}{Theorem}[section]

\newtheorem{lem}{Lemma}[section]

\theoremstyle{definition}
\newtheorem{defn}{Definition}[section]

\newtheorem{notn}{Notation}[section]

\newtheorem{quest}{Question}[section]
\newtheorem{rem}{Remark}[section]

\newtheorem{obs}{Observation}[section]

\makeatletter
\let\c@obs=\c@thm
\let\c@cor=\c@thm
\let\c@prop=\c@thm
\let\c@lem=\c@thm
\let\c@prob=\c@thm
\let\c@con=\c@thm
\let\c@conj=\c@thm
\let\c@defn=\c@thm
\let\c@notn=\c@thm
\let\c@notns=\c@thm
\let\c@exmp=\c@thm
\let\c@ax=\c@thm
\let\c@pro=\c@thm
\let\c@ass=\c@thm
\let\c@warn=\c@thm
\let\c@rem=\c@thm
\let\c@sch=\c@thm
\let\c@equation\c@thm
\let\c@quest=\c@thm
\numberwithin{equation}{section}
\makeatother

\newcommand{\FF}{\mathcal{F}}

\title{Graph Reconstruction from Connected Triples}
\author{Yaxin Qi}
\date{July 2023}

\begin{document}

\maketitle

\begin{abstract}
    The problem of graph reconstruction has been studied in its various forms over the years. In particular, the \emph{Reconstruction Conjecture}, proposed by Ulam and Kelly in 1942, has attracted much research attention and yet remains one of the foremost unsolved problems in graph theory. Recently, Bastide, Cook, Erickson, Groenland, Kreveld, Mannens, and Vermeulen proposed a new model of partial information, where we are given the set of connected triples $T_3$, which is the set of
    3-subsets of the vertex set that induce connected subgraphs. They proved that reconstruction is unique within the class of triangle-free graphs, 2-connected outerplanar graphs, and maximal planar graphs. They also showed that almost every graph can be uniquely reconstructed from their connected triples. However, little is known about other classes of non-triangle-free graphs within which reconstruction can occur uniquely, nor do we understand what kind of graphs can be uniquely reconstructed from their connected triples without assuming anything about the classes of graphs to which they belong.

    The main result of this paper is a complete characterization of all graphs that can be uniquely reconstructed from their connected triples $T_3$. We also show that reconstruction from $T_3$ is unique within the class of regular planar graphs, 5-connected planar graphs, certain strongly regular graphs, and complete multi-partite graphs, whereas it is not unique for the class of $k$-connected planar graphs with $k \leq 4$, Eulerian graphs, or Hamiltonian graphs.

\end{abstract}

\section{Introduction}
  The problem of graph reconstruction that asks how much of an unknown underlying graph of interest can be uniquely determined by specific types of information has been studied in various forms over the years \cite{N17} \cite{Cor09}. In particular, the \emph{Reconstruction Conjecture} \cite{SM60} \cite{K42}, proposed by Ulam and Kelly in 1942, has attracted much research attention and yet remains one of the foremost unsolved problems in graph theory \cite{Sca16}, \cite{W79} \cite{Hem77}, \cite{F74}. Recently, Bastide, Cook, Erickson, Groenland, Kreveld, Mannens, and Vermeulen \cite{BCEGKMV23} proposed a new model of partial information, where we are given the set of connected triples $T_3$, which is the set of 3-subsets of the vertex set that induce connected subgraphs. More generally, Bastide et al. gave the following definition.

\begin{defn}
    \cite{BCEGKMV23} For $k \geq 2$ and a finite, simple, connected (labeled) graph $G$, the \emph{set of connected $k$-sets} of $G$, which we will denote as $T_k(G)$, is defined to be \[ T_k(G) \coloneqq \{X \subseteq V(G) \mid |X| = k \text{ and the subgraph of $G$ induced by $X$ is connected}\}. \]
    In particular, when $k = 3$, we will call $T_3(G)$ \emph{the set of connected triples} of $G$.
\end{defn}
\begin{notn}
    We will use $N(v)$ to denote the set of neighbors of vertex $v \in V(G)$ whereas $N[v]$ denotes $N(v) \cup \{v\}$. We refer to the subgraph of $G$ induced by $X \subseteq V(G)$ as $G[X]$. In particular, we refer to $G[V(G) \setminus \{v\}]$ as $G-v$. Sometimes we drop $G$ if the graph in question is clear or if we don't want to emphasize the graph that gave rise to the set of connected triples.
\end{notn}

\begin{rem}
     When $k =2$, we get back our edge set and thus obtain the entire graph.
\end{rem}

As always, we are interested in when reconstruction of the underlying (labeled) graph is unique. The observation below establishes a connection among the connected $k$-sets for different values of $k$ and shows that the set of connected triples gives the most information about a graph among all the other non-trivial (i.e. $k \neq 2$) connected $k$-sets.

\begin{obs}
\cite{BCEGKMV23} For $k' \geq k \geq 2$, the connected $k-$sets of a graph are determined by the connected $k'$-sets. In particular, a $(k+1)$-set $X = \{x_1, ...,x_{k+1}\} \subseteq V(G)$ induces a connected subgraph of $G$ if and only if both $G[X \setminus \{y\}]$ and $G[X \setminus \{z\}]$ are connected for some $y,z \in X$.
\end{obs}

For this reason, we will focus on reconstruction of graphs from $T_3$ and make the following definition.

\begin{defn}
     A class $\mathcal{C}$ of graphs is $T_3$-reconstructible if all $G_1 \neq G_2 \in \mathcal{C}$ satisfies $T_3(G_1) \neq T_3(G_2)$. In other words, given $T_3(G)$ and the knowledge that the underlying graph $G$ is in $\mathcal{C}$, we are able to uniquely reconstruct $G$.
\end{defn}

\begin{rem}
    Since we are dealing with labeled graphs, we consider a graph unique if no other labeled graph is identical to it. For instance, the path $v_1v_2v_3$ would not be identical to the path $v_1v_3v_2$. Oftentimes the recognition problem of whether a graph $G$ is in a class $\mathcal{C}$ of graphs cannot be solved even from $T_3$. For example, we cannot distinguish a complete graph from an ``almost" complete graph—say a complete graph with an arbitrary edge deleted—because they would both have all $3$-subsets of the vertex set as their connected triples. In the case where the order of the graph is $4$, we cannot even distinguish a complete graph from a cycle.
\end{rem}

Bastide et al. \cite{BCEGKMV23} proved that $T_3$-reconstructible classes of graphs include triangle-free graphs on $n \geq 5$ vertices, outerplanar 2-connected graphs on $n \geq 6$ vertices, and maximal planar graphs on $n \geq 7$ vertices. They also showed that almost every graph can be uniquely reconstructed from their connected triples without assuming additional information. However, little is known about other classes of non-triangle-free graphs that are $T_3$-reconstructible, nor do we understand what kind of graphs can be uniquely reconstructed from their connected triples without assuming anything about the classes of graphs to which they belong.

The main result of this paper is a complete characterization of all graphs that can be uniquely reconstructed from their connected triples. We also show that regular planar graphs, 5-connected planar graphs, certain strongly regular graphs, and complete multi-partite graphs are reconstructible, whereas $k$-connected planar graphs for $k \leq 4$, Eulerian graphs, and Hamiltonian graphs are not.

This paper is organized as follows. In Section 2 we present a few preliminary results not necessarily restricted to reconstruction from connected triples. From Section 3 and onward, we focus exclusively on answering the question of when reconstruction from connected triples is unique. In particular, in Section 3 we prove or disprove the $T_3$-reconstructibility of several families of non-triangle free graphs. In Section 4 we study graphs that can be uniquely reconstructed from their set of connected triples without assuming any additional information, establishing a complete characterization from scratch. Finally in Section 5 we give some future directions.

\section{Preliminary Results}
Generally, two distinct labeled graphs can have the same connected triples. So if we are only given $T_3$ and the order of a graph $|V(G)| = n$, we cannot always expect to uniquely reconstruct the underlying graph even from its connected triples. However, if we are given the order of $G$, its vertex connectivity $\kappa(G)$ can be uniquely reconstructed from $T_k(G)$ for most values of $k$ that are ``not too large." Although it is easier to prove this using Observation 1.4, we prove it in a way that will help set up the proof of the next result regarding complete multipartite graphs.

\begin{rem}
    When we say that ``$T_k(G)$ uniquely determines $\kappa(G)$" in the context of the following theorem, what we really mean, of course, is that if $|V(G)| = n$ and $2 \leq k \leq n - \kappa(G)$ are fixed, then for all graphs $G'$ satisfying $|V(G')| = n$, we have $T_k(G) = T_k(G')$ implies $\kappa(G') = \kappa(G)$.
\end{rem}

\begin{thm}
    If the order of a graph $|V(G)| = n$ is known, then for all $2 \leq k \leq n - \kappa(G)$, where $\kappa(G)$ denotes the vertex connectivity of $G$, we have that $T_k(G)$ uniquely determines $\kappa(G)$.
\end{thm}
\begin{proof}
    We define a subset $S \subseteq T_k(G)$ to be a \emph{gluing set of $G$} if for all $s_1 \in S$, there exists $s_2 \neq s_1$ such that $s_1 \cap s_2 \neq \emptyset$. We call $g(S) \coloneqq \bigcup_{s \in S}s$ a \emph{glued set of $G$}. It is clear that a graph $H$ with order $m$ is connected if and only if there exists a glued set $g(S)$ of $H$ with cardinality $m$ for some gluing set $S \subseteq T_k(H)$. Indeed, if $H$ were disconnected, then each glued set would be contained in one of the components of $H$. On the other hand, if $H$ were connected, take a maximal glued set, say $g(S)$, and a vertex $v \in g(S)$. For all vertices $u \neq v \in V(G)$, there exists a path from $u$ to $v$. If the path has length at least $k$, then it can be chopped up into overlapping paths on $k$ vertices whose vertex sets all belong to $S$. Otherwise, $u$ and $v$ would exist in a connected subgraph of $G$ with $k$ vertices. Since $S$ gives rise to a maximal glued set, $u$ would be in some set in $S$ and thus in $g(S)$. To uniquely determine $\kappa(G)$, we first drop one vertex at a time and check if all of $v_i \in V(G)$ satisfies that $G-v_i$ has a glued set with order $m-1$. If the answer is yes, we move on to removing two vertices at a time and checking whether all of the corresponding graphs have a glued set with order $m-2$. It is worth noting that when we drop a set of vertices and check for glued sets of the appropriate size, we temporarily delete elements of $T_k(G)$ that contain any of those vertices. We can uniquely determine $\kappa(G)$ by continuing this process and finding the smallest $k$ such that removing $k$ vertices at a time gives a corresponding graph with no glued set of cardinality $n-k$.
\end{proof}

\begin{thm}
    For all complete $n$-partite graphs $K_{r_1,...,r_n}$, where $n \geq 3$ and $r_i \geq 3$ for all $i$, $K_{r_1,...,r_n}$ can be uniquely reconstructed from its connected $k$-sets $T_k$ for all $k \leq \min \{r_1,...,r_n\}$ if we know that it is a complete multi-partite graph.
\end{thm}
\begin{proof}
     Observe that a $k$-subset of $V(K_{r_1,...,r_n})$ does not induce a connected subgraph if and only if it is contained within one partite of $K_{r_1,...,r_n}$. Take the complement $\overline{T_k}$ of $T_k$, where \[ \overline{T_k(G)} \coloneqq \{X \subseteq V(G) \mid |X| = k \text{ and } G[X] \text{ is not connected}\}, \] and look at the glued sets of $\overline{T_k}$. They form $n$ chains. Taking the upper bounds of the chains give the $n$ partites.
\end{proof}

From now on, we will focus on reconstruction of graphs from $T_3$ exclusively. We prove one more preliminary result in this section.

\begin{defn}
    A strongly regular graph with parameters $G = (v,k, \lambda, u)$ is a $k$-regular graph on $v$ vertices, where every two adjacent vertices share exactly $\lambda$ common neighbors and any two distinct non-adjacent vertices share exactly $u$ common neighbors.
\end{defn}

\begin{thm}
    For all strongly regular graphs $G = (v,k, \lambda, u)$ that satisfy $2k-\lambda \neq u +2$ and $v \neq 2k+1$, if we know the value of $k$ and that $G$ is strongly regular, then $G$ can be uniquely reconstructed from $T_3(G)$.
\end{thm}
\begin{proof}
    For all $v_1 \neq v_2 \in V(G)$, if $v_1v_2 \in E(G)$, then the number of connected triples in $T_3(G)$ that contain both $v_1$ and $v_2$ is $2(k-1)-\lambda$, which is different from $u$, the number of connected triples in $T_3(G)$ that contain two nonadjacent vertices. So pairs of vertices that appear in the same number of connected triples are either all adjacent or all non-adjacent, resulting in only two possible edge assignments. The only difficulty is that we don't know which of the two numbers is $u$ and which is $2(k-1)-\lambda$. Nevertheless we can temporarily assign $uv$ to be an edge for all $u \neq v \in V(G)$ that are contained in one number of connected triples and all the other pairs that are contained in the other number of connected triples as non-edges. Then pick a vertex $v \in V(G)$, and check if it is contained in exactly $k$ of our assigned edges. If so, our edge assignment was correct. Otherwise, the only other possible edge assignment obtained by flipping all our current edges and non-edges is correct.
\end{proof}

\section{More $T_3$-Reconstructible Classes of Graphs}

In this section, we prove two results regarding the $T_3$-reconstructibility of the class of $k$-connected planar graphs and the class of regular planar graphs. We will do so by finding information from $T_3$ that almost gives the set of neighbors and then identifying the fake neighbors. We also present a construction that shows that the class of Hamiltonian graphs and the class of Eulerian graphs are not $T_3$-reconstructible.

\begin{defn}
    The \emph{set of roughly neighbor sets $\mathcal{N}(v)$} of a vertex $v \in V(G)$ is a set whose elements, which we will refer to as $T_3$-neighborhoods, are exactly the maximal subsets of $V(G)$ that satisfy the property that for all $v_1 \neq v_2$ in a $T_3$-neighborhood ${N_v}$, we have $\{v_1, v_2, v\} \in T_3(G)$.
\end{defn}

Clearly, for every vertex $v \in V(G)$, we can always uniquely determine $\mathcal{N}(v)$ from $T_3(G)$ by, in the worst case scenario, checking each subset of $V(G)$ to see if every pair of its elements form a connected triple with $v$ and then collecting all such distinct sets with maximal cardinality to get $\mathcal{N}(v)$.

\begin{obs}
     If $\mathcal{N}(v)$ has exactly one element, then either it is exactly the set of neighbors of $v$, which we denote as $N(v)$, or it is $N(v) \cup \{w\}$, where $w \notin N[v]$ is the unique vertex that is adjacent to everything in $N(v)$. If $\mathcal{N}(v)$ has multiple elements, then either they are all of the form of $N(v) \cup \{w_i\}$, for some $w_i \notin N[v]$ that is adjacent to everything in $N(v)$, or there is one element that is exactly $N(v)$ and the rest of the elements are of the form $\left(N(v)\setminus \{v_i\} \right) \cup \{w_j\}$ for some $v_i \in N(v)$ and $w_j$ that is adjacent to everything in $N(v)\setminus \{v_i\}$.
\end{obs}

Now we prove that the class of 5-connected planar graphs is $T_3$-reconstructible, whereas the class of $k$-connected planar graphs is not for $k \leq 4$. We start by proving two lemmas.

 \begin{lem}
        No two adjacent vertices in a 5-connected planar graph $G$ can have more than three vertices in common.
    \end{lem}
    \begin{proof}
        Suppose, for contradiction, that there exist two adjacent vertices $v \neq v_i \in V(G)$ with distinct common neighbors $v_1, v_2, v_3,$ and $v_4$. Since $G$ is 5-connected, there is a path from $v_1$ to $v_2$ that does not contain $v_3, v_4, v$, or $v_i$, which we denote as $P_{v_1v_2}$. Similarly there exist $P_{v_2v_4}$ and $P_{v_1v_4}$ that does not contain any vertex from $\{v_1, v_3, v_i, v\}$ and $\{v_2, v_3, v, v_i\}$, respectively. If any one of the three paths, say $P_{v_1v_2}$, is internally vertex-disjoint with the other two, then we can contract all the edges except the one incident to $v_2$ in $P_{v_2v_4}$ in a way that effectively glues all of its internal vertices to $v_4$, and similarly we glue all the internal vertices of $P_{v_1v_4}$ to $v_4$. After contracting $P_{v_1v_2}$ to an edge between $v_1$ and $v_2$, we notice that $G[\{v, v_i, v_2, v_4, v_1\}]$ contains a $K_5$ minor, which contradicts the premise that $G$ is planar. Otherwise, without loss of generality, assume both $P_{v_1v_2}$ and $P_{v_2v_4}$ intersect $P_{v_1v_4}$ at some internal vertex, respectively. Let $a$ be the vertex in $P_{v_1v_2} \cap P_{v_1v_4}$ that is closet to $v_2$ in $P_{v_1v_2}$. Contract all edges in $P_{v_1v_4}[v_1-a]$, which is the part of $P_{v_1v_4}$ that starts at $v_1$ and ends at $a$, so that $P_{v_1v_4}[v_1-a]$ becomes just the edge $v_1a$. Then, contract $P_{v_1v_4}[a-v_4]$ to $av_4$ and $P_{v_1v_2}[a-v_2]$ to $av_2$. Call the resulting graph $\Tilde{G}$. Now notice that $\Tilde{G}[\{v,a,v_i, v_1, v_2, v_4\}]$ contains a $K_{3,3}$ as its subgraph. This shows that the original graph $G$ contains a $K_{3,3}$ minor and thus can not be planar. This is a contradiction.
    \end{proof}

    \begin{lem}
        For a 5-connected planar graph $G$ and a vertex $v \in V(G)$, a $T_3$-neighborhood $N_v$ of $v$ is exactly the set of neighbors $N(v)$ if and only if $N_v$ does not contain a vertex $w$ with a $T_3$-neighborhood $N_w$ that contains $N_v \setminus \{w\}$. In particular, no $T_3$-neighborhood of a neighbor $v_i$ of $v$ can contain $N(v) \setminus \{v_i\}$.
    \end{lem}
    \begin{proof}
        It follows from Observation 3.2 that if an element ${N_v}$ in $\mathcal{N}(v)$ is not $N(v)$, then it must contain an element $w$ that is adjacent to everything in ${N_v} \setminus \{w\}$.  Since $N(w)$ will always be contained in some element, say ${N_w}$, of $\mathcal{N}(w)$, we have that ${N_w} \in \mathcal{N}(w)$ contains ${N_v} \setminus \{w\}$.

        For the other direction, suppose, for contradiction, that ${N_v} = N(v)$ contains an element $v_i$ such that there exists ${N_{v_i}} \in \mathcal{N}(v_i)$ that contains $N(v) \setminus \{v_i\}$. Then either $v_i$ is adjacent to everything in $N(v) \setminus \{v_i\}$ or it is adjacent to everything except say $v_h$, in which case $v_h$ would need to be adjacent to at least $\operatorname{deg}(v_i)-1$ many vertices in $N(v_i)$ by Observation 3.2. Since $G$ is 5-connected, every vertex including $v$ would have degree at least 5. Hence the first scenario is impossible because we know by Lemma 3.3 that $v$ and $v_i$ cannot be adjacent and share more than three common neighbors in a planar graph.

        So $v_i$ is adjacent to $N(v) \setminus \{v_i, v_l\}$ for some $v_l \neq v_i \in N(v)$. If $|N(v)| \geq 6$, then $v_i$ and $v_l$ would share at least four common neighbors within $N(v)$, which, together with $v, v_i$, and $v_l$ would induce a subgraph that contains a $K_{3,3}$, a contradiction. Finally supoose $|N(v)| = 5$ and say $N(v) = \{v_i, v_1, v_2, v_3, v_l\}$. We know $v_l$ cannot be adjacent to all of $v_1, v_2,$ and $v_3$ because otherwise $G[\{v_i, v_l, v, v_1, v_2, v_3\}]$ would contain a $K_{3,3}$, which is a contradiction. Yet the fact that $v_l$ is adjacent to at least $\operatorname{deg}(v_i)-1$ many vertices in $N(v_i)$ implies that $v_l$ is adjacent to two of $v_1, v_2, v_3$, say they are $v_2$ and $v_3$, and shares at least one common neighbor $x$ with $v_i$ outside of $N(v)$. But since $G$ is 5-connected, we know there is a path between $v_l$ and $v_1$ that does not go through any of the vertices in $\{v, v_i, v_2, v_3\}$. Therefore, after we contract this path to a single edge between $v_l$ and $v_1$, the vertex set $\{v,v_i, v_l, v_1, v_2 ,v_3\}$ will induce a subgraph in the resulting graph that contains a $K_{3,3}$. This means that $G$ has a $K_{3,3}$ minor, which is a contradiction.
    \end{proof}

\begin{thm}
    For all $k \leq 4$, the class of $k$-connected planar graphs is not $T_3$-reconstructible, whereas the class of $5$-connected planar graphs is $T_3$-reconstructible and the class of $k$-connected planar graphs does not exist for $k \geq 6$.
\end{thm}
\begin{proof}
    First, it is clear that if a graph is $k$-connected, then every vertex must have degree at least $k$, otherwise removing the neighbors of a vertex that has degree less than $k$ will disconnect the graph. So a $k$-connected planar graph, where $k \geq 6$, requires every vertex to have degree at least 6, which would give at least $3n$ edges, which is strictly greater than $3n-6$, the maximum number of edges a planar graph on $n$ vertices can have. This is a contradiction. Thus $k$-connected planar graphs do not exist for $k \geq 6$.

    As for why the class of $k$-connected planar graphs is not $T_3$-reconstructible for all $k \leq 4$, we present an arbitrarily large construction of two 4-connected planar graphs, obtained by switching the labels of $n-1$ and $n-2$, that are not identical as labeled graphs but share the same set of connected triples in Figure~\ref{fig:construction}.

    We next show that the class of 5-connected planar graphs is $T_3$-reconstructible by showing that we can uniquely determine the neighbors of every vertex. Fix a vertex $v \in V(G)$ and look at $\mathcal{N}(v)$. Regardless of which of the four cases mentioned in Observation 3.2 we have, we would be able to identify $N(v)$ if we can identify the existence of $w$ or $w_i$ and precisely which vertex it is.

    With Lemma 3.4, we can pick out $N(v)$ if it is contained in $\mathcal{N}(v)$. Otherwise, every element of $\mathcal{N}(v)$ must be of the form $N(v) \cup \{w_i\}$, where $N(v) \subseteq N(w_i)$. We can pick a random $N_v = N(v) \cup \{w_i\}$ from $\mathcal{N}(v)$. We know from Lemma 3.4 that $w_i$ is the only vertex in $N_v$ with a $T_3$-neighborhood that contains $N(v)$, since for any $s \in N(v)$, a $T_3$-neighborhood $N_s$ cannot even contain $N(v) \setminus \{s\}$. Thus we can identify $w_i$ as the vertex with a $T_3$-neighborhood that contains all but one element of $N_v$, while all $T_3$-neighborhoods of neighbors of $v$ are necessarily missing at least two elements of $N_v$. Throwing out the $w_i$ from $N_v = N(v) \cup \{w_i\}$ gives us $N(v)$.
\end{proof}

\begin{figure}[htbp]
\centering
\includegraphics{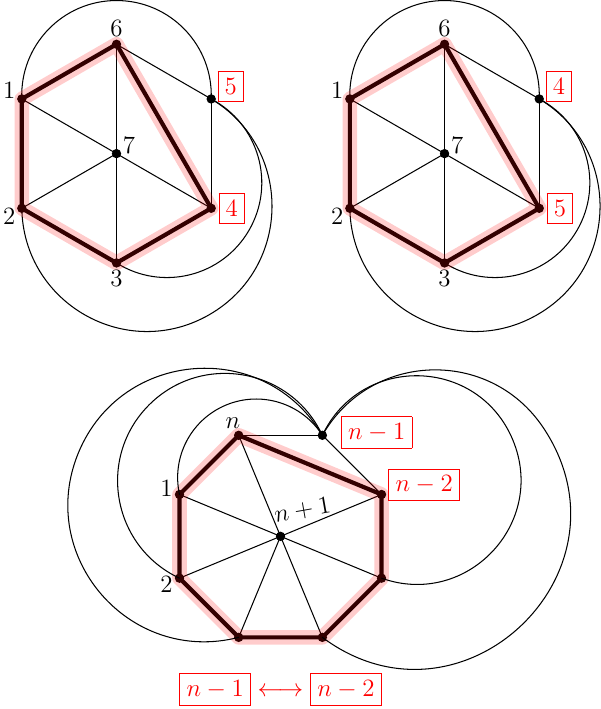}
\caption{An arbitrarily large construction that shows the class of $k$-connected planar graphs is not $T_3$-reconstructible for $k \leq 4$. }
\label{fig:construction}
\end{figure}

\begin{rem}
    Since the two arbitrarily large graphs in Figure~\ref{fig:construction} obtained by switching the labels of $n-1$ and $n-2$ are both Hamiltonian and Eulerian, this construction also shows that the class of Hamiltonian graphs and the class of Eulerian graphs are not $T_3$-reconstructible.
\end{rem}

Next, we show that the class of regular planar graphs on $n \geq 7$ vertices is $T_3$-reconstructible. To do this, we show that we can first recognize the degree $d$ and then show that we can uniquely reconstruct the underlying graph $G$ knowing that it is a $d$-regular planar graph. Note that $d$ can only range from two to five since we are dealing with connected planar graphs. We break down the proof of the theorem into a series of lemmas and their proofs below.

\begin{lem}
        If $G$ is a 3-regular planar graph, then there exists a vertex $v \in V(G)$ such that every element in $\mathcal{N}(v)$ has order 3.
    \end{lem}
    \begin{proof}
        Suppose, for contradiction, that this is not the case. Then for every $v \in V(G)$, there exists $w \neq v \in V(G)$ such that $N(v) = \{v_1, v_2, v_3\} = N(w)$. Fix $v,w$ that satisfy the above. If $\{v_1, v_2, v_3\}$ forms an independent set, then $v_1$ would have a neighbor, say $x_1$, outside of $N(v) \cup \{v,w\}$. Then only $v_2$ or $v_3$ can have the same exact set of neighbors as $v_1$. Without loss of generality suppose it's $v_3$. If $v_2$ is also a neighbor of $x_1$, then $G[\{v,w, x_1, v_1, v_2, v_3\}]$ would contain a $K_{3,3}$, a contradiction. Thus $v_2$ has a neighbor $x_2 \neq x_1$ outside of $N(v) \cup \{v,w\}$. But neither $v_1$ nor $v_3$ can be adjacent to $x_2$ because they both have degree three. So there does not exist $u \neq v_2 \in V(G)$ with $N(v_2) = N(u)$, which is a contradiction. If $\{v_1, v_2, v_3\}$ does not form an independent set, however, then there would exist at least one edge, say $v_1v_2$, among $\{v_1, v_2, v_3\}$. In this case, there won't exist any $u \neq v_3 \in V(G)$ that satisfies $N(u) = N(v_3)$ because the only possible candidates are $v_1$ and $v_2$, both of which already have three neighbors. This is a contradiction.
    \end{proof}

    \begin{lem}
        The class of 3-regular planar graphs on greater or equal to 5 vertices is $T_3$-reconstructible.
    \end{lem}
    \begin{proof}
         If $\mathcal{N}(v)$ has only one element $N_v$, then $N(v) = N_v$ if $|N_v| = 3$, otherwise $N_v = N(v) \cup \{w\}$. We will show that we can recognize $w$ in the latter case by showing it is the only vertex $s$ in $N_v$ such that there exists ${N_s}^j \in \mathcal{N}(s)$, where ${N_s}^j = (N_v \cup \{v\}) \setminus \{s\}$. Let $N(v) = \{v_1, v_2, v_3\}$. Suppose, for contradiction, that say $v_1 \in N(v)$ is also such that there exists ${N_{v_1}}^j \in \mathcal{N}(v_1)$, where ${N_{v_1}}^j = (N_v \cup \{v\}) \setminus \{v_1\}$. Then $v_1$ would need to be adjacent to one of $v_2$ and $v_3$ and share all its neighbors with the other, which would result in one of $v_2$ and $v_3$ having degree four, a contradiction.

         On the other hand, if $\mathcal{N}(v)$ has multiple elements, it cannot be the case that they are each of the form $N(v) \cup\{w_i\}$, where $N(w_i) = N(v)$, otherwise $G[N(v) \cup \{v, w_i, w_j\}]$ would contain a $K_{3,3}$, which is a contradiction. By Observation 3.2, we know that $N(v)$ is an element in $\mathcal{N}(v)$ and all the other elements would be of the form $S_j \coloneqq (N(v) \setminus \{v_i\}) \cup \{w_j\}$, where $w_j \notin N(v)$ and $N(v) \setminus \{v_i\} \subset N(w_j)$. If there is exactly one element $N_v \in \mathcal{N}(v)$ such that $|N_v \cap S_j| = 2$ for all $S_j \neq N_v \in \mathcal{N}(v)$, then we have $N(v) = N_v$.

         Otherwise, we have $\bigcap_{S_j \in \mathcal{N}(v)} S_j = \{v_1, v_2\} \subset N(v)$ and we just need to distinguish $v_3 \in N(v) \setminus \{v_1, v_2\}$ from all the $w_j$. Note that there can be at most two such $w_j$, say $w_1$ and $w_2$, because of the 3-regularity condition, and that $v_3$ is only contained in the element $N(v) \in \mathcal{N}(v)$. Suppose there are three elements in $\mathcal{N}(v)$ and thus both such $w_1$ and $w_2$ exist. Then $v_3$ cannot be adjacent to $v_1$ or $v_2$, and thus $\{v_1, v_2\}$ is a subset of some element in $\mathcal{N}(w_1)$ and $\mathcal{N}(w_2)$ but not of $\mathcal{N}(v_3)$.

         Now, suppose $\mathcal{N}(v)$ contains two elements and there exists only one of such $w_1$ and $w_2$ mentioned above. Without loss of generality, let it be $w_1$. If there exists $x \in V(G)$ with $N(x) = N(w_1)$, which happens if and only if elements in $\mathcal{N}(w_1)$ would all have cardinality four, then $v$ would not be contained in any element of $\mathcal{N}(w_1)$. Thus in this case we can distinguish $w_1$ from the neighbors of $v$. If no such $x$ exists, then the cardinality of elements in $\mathcal{N}(w_1)$ is three and at least two of them have intersection $\{v_1, v_2\}$. We can easily check that if the same were true for $v_3$, then $v_3$ would have to be adjacent to both $v_1$ and $v_2$. Then note that $\mathcal{N}(w_1) = \{ N(w_1), \{v_1, v_2, v\}, \{v_1, v_2, v_3\} \}$ would have three elements whereas $\mathcal{N}(v_3) = \{ N(v_3) = \{v,v_1,v_2\}, \{v_1, v_2, w_1\} \}$ would only have two. Thus, we can differentiate between $v_3$ and $w_1$.
    \end{proof}

 \begin{lem}
        If $G$ is a 4-regular planar graph on $n \geq 7$ vertices, then there exists vertex $v \in V(G)$ such that every element in $\mathcal{N}(v)$ has order 4.
    \end{lem}
    \begin{proof}
        Suppose, for contradiction, that this is not the case. Then for every $v \in V(G)$, there exists $w \neq v \in V(G)$ such that $N(v) = \{v_1, v_2, v_3, v_4\} = N(w)$. Fix a $v \in V(G)$ and let $w \neq v \in V(G)$ be such that $N(w) = N(v) = \{v_1, v_2, v_3, v_4\}$.

        If there exists $v_1 \in N(v)$ that is adjacent to two other vertices in $N(v)$, say $v_2$ and $v_3$, then $v_4$ is the only vertex that can play the role of satisfying $N(v_4) = N(v_1)$. Hence $G[\{v,w, v_1, v_2, v_3, v_4\}]$ would give the unique (up to isomorphism) 4-regular planar graph on 6 vertices and thus forms a component of $G$. But $G$ has $n \geq 7$ vertices, which means that it will have more than one component, a contradiction.

        Otherwise, suppose there exists $v_1 \in N(v)$ that is adjacent to one other vertex in $N(v)$, say $v_2$. One of $v_3$ and $v_4$ has to have the same set of neighbors as $v_1$. Without loss of generality, assume it's $v_3$. But now all vertices in $\{v, w, v_1, v_2, v_3\}$ have four neighbors already. Then clearly, no vertex other than $v_4$ can have the same set of neighbors as does $v_4$, which is a contradiction.

        If not, then no two vertices in $N(v)$ can be adjacent to each other. Then $N(v)$ can be partitioned into pairs that have the same set of neighbors. Suppose $N(v_1) = \{v, w, x_1, x_2\} = N(v_4)$ and $N(v_2) = \{v, w, y_1, y_2\} = N(v_3)$. Observe that we would necessarily have $N(x_1) = N(x_2)$ and $N(y_1) = N(y_2)$. If $y_1$ and $y_2$ are adjacent to $x_1$ and $x_2$, then, after contracting the path $v_4vv_3$ to $v_4v_3$, $G$ contains a $K_{3,3}$ minor with the vertex set being $\{y_1, y_2, v_3, v_4, x_1, x_2\}$, a contradiction. So we need distinct new vertices $y_3$ and $y_4$ to be the other two neighbors of both $y_1$ and $y_2$. If $x_1$ and $x_2$ are adjacent to $y_3$ and $y_4$, then contracting the path $v_4vv_2y_2$ to $v_4y_2$ gives that $G$ contains a $K_{3,3}$ minor with the vertex set being $\{y_3, y_4, v_4, y_2, x_1, x_2\}$, which is also a contradiction. Hence we also need new vertices, say $x_3$ and $x_4$, to be the last two neighbors of both $x_1$ and $x_2$. Applying similar arguments repeatedly shows that we will always need more new vertices to be the last two neighbors of $x_i, x_{i+1}$ and $y_j, y_{j+1}$, respectively. This gives a contradiction because our graph is finite.
    \end{proof}

    \begin{lem}
        The class of 4-regular planar graphs on greater or equal to 7 vertices is $T_3$-reconstructible.
    \end{lem}
    \begin{proof}
        Fix a vertex $v \in V(G)$. Suppose $\mathcal{N}(v)$ has only one element $N_v$. Then $N(v) = N_v$ if $|N_v| = 4$. Otherwise, $N_v = N(v) \cup \{w\}$ for the unique $w \notin N(v)$ with $N(v) = N(w)$. Note that $w$ satisfies that $(N_v \cup \{v\}) \setminus \{w\} \in \mathcal{N}(w)$. If there exists $v_1 \in N(v)$ that also satisfies this, then $v_1$ would be adjacent to two other vertices in $N(v)$. As we have shown in the proof of Lemma 3.9, this leads to $\{N(v) \cup \{v, w\} \}$ forming a connected component of size 6, which implies that $G$ with $n \geq 7$ vertices would not be connected, a contradiction. Hence we can recognize $w$ and thus determine $N(v)$ in this case.

        On the other hand, suppose $\mathcal{N}(v)$ has multiple elements. Then $N(v)$ is an element in $\mathcal{N}(v)$ and all the other elements would be of the form $S_j \coloneqq (N(v) \setminus \{v_i\}) \cup \{w_j\}$, where $w_j \notin N(v)$ and $N(v) \setminus \{v_i\} \subset N(w_j)$. The other case for multiple elements recognized in Observation 3.2 would give a $K_{3,3}$ minor, a contradiction. If $\mathcal{N}(v)$ has $\geq 3$ elements, then we can recognize $N(v)$ as the unique element that has an intersection of size three with all other elements, respectively. Now suppose $\mathcal{N}(v)$ has exactly two elements: $N(v) = \{v_1, v_2, v_3, v_4\}$ and $S \coloneqq (N(v) \setminus \{v_4\}) \cup \{w\}$. To recognize $N(v)$ is equivalent to recognizing the existence of $w$.

        We observe that $\mathcal{N}(w)$ has multiple elements of size four, contains $(S \cup \{v\}) \setminus \{w\}$ as an element, where $S \in \mathcal{N}(v)$, and $v$ is contained in exactly one element of $\mathcal{N}(w)$. Suppose, for contradiction, that the same holds for some $v_1 \in N(v)$. Then either $v_1$ is adjacent to all other vertices in $N(v)$ or all other vertices except for some $v_4 \in N(v)$. The former cannot be possible because for $v$ to exist in exactly one element of $\mathcal{N}(v_1)$ that has multiple elements, there has to exist a vertex $x \neq v_1, v$ that is adjacent to everything in $N(v) \setminus \{v_1\}$, which would result in a $K_{3,3}$ minor, a contradiction. As for the latter, $v_4$ needs to be adjacent to at least three vertices in $N(v_1)$ to be included in some element of $\mathcal{N}(v_1)$. If $v_4$ is adjacent to all four vertices in $N(v_1)$, then elements of $\mathcal{N}(v_1)$ would have size five and not four, which is a contradiction. So $v_4$ is adjacent to three vertices in $N(v_1)$. But in this case $v$ would be contained in both $(N(v_1) \cap N(v_4)) \cup \{v_4\}$ and $N(v_1)$ in $\mathcal{N}(v_1)$, which is also a contradiction.
    \end{proof}

    \begin{lem}
        The class of 5-regular planar graphs on greater or equal to 5 vertices is $T_3$-reconstructible.
    \end{lem}
    \begin{proof}
        Observe that if $\mathcal{N}(v)$ has multiple elements, then it has to have exactly two elements: $S_1 \coloneqq N(v)$ and $S_2 \coloneqq (N(v) \setminus \{v_1\}) \cup \{w\}$ for some $v_1 \in N(v)$ and $w \notin N(v)$ with $(N(v) \setminus \{v_1\}) \subset N(w)$, otherwise our planar graph would contain a $K_{3,3}$ minor, a contradiction. It's easy to check that $\mathcal{N}(w)$ has multiple elements, including $(S \cup \{v\}) \setminus \{w\}$ for some $S \in \mathcal{N}(v)$ that contains $w$, and that $v$ would belong to exactly one element in $\mathcal{N}(w)$. However, the same cannot be all true for any $v_i \in N(v)$. Thus we will be able to determine the existence of $w$ in elements of $\mathcal{N}(v)$ and pick out $N(v)$.

        Now suppose $\mathcal{N}(v)$ has only one element $N_v$. If $N_v$ has size 5, then we have $N(v) = N_v$. Else, $N_v = N(v) \cup \{w\}$, where $N(w) = N(v)$. We want to be able to recognize $w$ in $N_v$. Clearly, we have that there exists $N_w \in \mathcal{N}(w)$ such that $N_w = (N_v \cup \{v\}) \setminus \{w\}$. Suppose, for contradiction, that there exists $v_i \in N(v)$ with $(N_v \cup \{v\}) \setminus \{v_i\} \in \mathcal{N}(v_i)$. Then $v_i$ needs to be adjacent to three vertices within $N(v) \setminus \{v_i\}$. But this would mean that $G$ contains a $K_{3,3}$ minor, which is a contradiction.
    \end{proof}

    Piecing together the lemmas above, we obtain a proof of Theorem 3.12 below.

\begin{thm}
    The class of regular planar graphs on greater or equal to 7 vertices is $T_3$-reconstructible.
\end{thm}
\begin{proof}

    If $T_3(G) = \{ \{v_{i}, v_{i+1}, v_{i+2} \} \mid i \in \mathbb{Z}/n\mathbb{Z} \}$, then $d =2$ and in particular $G$ is the cycle $v_0v_1v_2...v_{n-2}v_{n-1}v_0$. This is proved in Observation 4.3.

    Otherwise if there exists vertex $v \in V(G)$ such that every element in $\mathcal{N}(v)$ has order three, then we know $d = 3$. And by Lemma 3.7, we will always be able to recognize when $d = 3$ this way.

    Otherwise if there exists vertex $v \in V(G)$ such that every element in $\mathcal{N}(v)$ has order four, then we know $d = 4$. Lemma 3.9 shows that we will always be able to recognize when $d = 4$ this way.

    Otherwise, $d = 5$. Lemma 3.11, together with Lemma 3.10 and Lemma 3.8, concludes the proof of Theorem 3.12.

\end{proof}

\section{Strongly $T_3$-Reconstructible Graphs}
In this section, we seek to directly answer the question of when reconstruction from connected triples is unique without allowing the additional information of to which classes of graphs our underlying graph of interest belongs. While Bastide et al. have shown that almost every graph can be uniquely reconstructed in this sense, essentially nothing is known about what kind of graphs these are and what it would take to satisfy this unique reconstruction condition. In this section, we explicitly define what we mean by being able to be uniquely reconstructed from $T_3$, give some examples of such graphs, provide a framework for checking if some graph can be uniquely reconstructed from $T_3$, and prove a series of lemmas—including a characterization of triangle-free graphs that can be uniquely reconstructed—that build up to a complete characterization of graphs that can be uniquely reconstructed from $T_3$.

\begin{defn}
    A finite, simple, connected labeled graph $G$ is \emph{strongly $T_3$-recon\\structible} if for any finite, simple, connected labeled graph $H$ that has the same set of connected triples as $G$, we have that $H$ is identical to $G$.
\end{defn}

\begin{quest}
    What kind of graphs are strongly $T_3$-reconstructible?
\end{quest}

The simplest examples are $l$-cycles and $l$-wheels for all $l \geq 5$, the proofs of which are very straightforward and can be found in \cite[Observation 4,5]{BCEGKMV23}. We will include them here for completeness.

\begin{obs}
    All cycles on $n \geq 5$ vertices are strongly $T_3$-reconstructible.
\end{obs}
\begin{proof}
    \cite[Observation 4]{BCEGKMV23} For any cycle $C_n$, we have $T_3(C_n) = \{ \{v_{i}, v_{i+1}, v_{i+2} \} \mid i \in \mathbb{Z}/n\mathbb{Z} \}$. Let $G$ be a graph with $T_3(G) = T_3(C_n)$. If there exists $i \in \mathbb{Z}/n\mathbb{Z}$ such that $v_{i}v_{i+1}$ is not an edge of $G$, then the fact that $\{v_{i-1}, v_{i}, v_{i+1}\}$ and $\{ v_{i}, v_{i+1}, v_{i+2} \}$ are in $T_3(C_n)$ implies that $v_{i-1}v_{i}, v_{i-1}v_{i+1}, v_{i}v_{i+2},$ and $v_{i+1}v_{i+2}$ are edges of $G$. Since $n \geq 5$, there exists $v_{i+3} \notin \{ v_{i-1}, v_{i}, v_{i+1}, v_{i+2}\}$ such that $v_{i+3}v_{i+2} \in E(G)$. This means that $\{v_{i}, v_{i+2}, v_{i+3}\} \in T_3(G) = T_3(C_n)$, which is a contradiction. On the other hand, suppose there exist $i,j \in \mathbb{Z}/n\mathbb{Z}$, where $j \neq 1$ or $n-1$, such that $v_{i}v_{i+j} \in E(G)$. Then both $\{v_{i}, v_{i+1}, v_{i+j}\}$ and $\{v_{i-1}, v_{i}, v_{i+j}\}$ are in $T_3(G) = T_3(C_n)$. But we would need $2 = j = n-2$ for this to happen, which is impossible for all $n \geq 5$.
\end{proof}

\begin{obs}
    All $l$-wheels are strongly $T_3$-reconstructible for $l \geq 5$.
\end{obs}
\begin{proof}
      \cite[Observation 5]{BCEGKMV23} A graph $G$ is a $l$-wheel if and only if there exists a vertex $v \in V(G)$ such that $G \setminus v$ is a $l$-cycle and $v$ appears in $\binom{l}{2}$ triples.
\end{proof}

\begin{obs}
    All paths on $n \geq 5$ vertices are strongly $T_3$-reconstructible.
\end{obs}
\begin{proof}
   This is a corollary of Theorem 3.8, which we prove later.
\end{proof}

We provide some examples of small graphs that are strongly $T_3$-reconstructible in Figure~\ref{fig:small}. Note that if a graph $G$ contains a strongly $T_3$-reconstructible graph $H$ as an induced subgraph, then we can uniquely reconstruct the embedding of $H$ in $G$ and go from there. \hfill
\break

\begin{figure}
\centering
\includegraphics{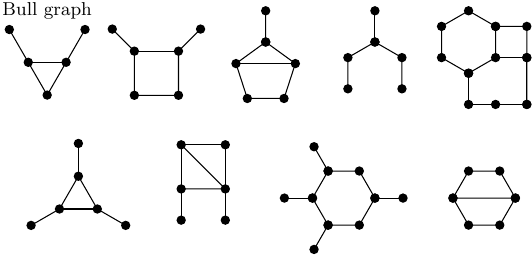}
\caption{Some examples of small graphs that are strongly $T_3$-reconstructible.}
\label{fig:small}
\end{figure}

\begin{defn}
    An edge $uv \in E(G)$ is \emph{necessary} if there does not exist any graph $H$ with $uv \notin E(H)$ and $T_3(H) = T_3(G)$.
\end{defn}

We state the following observation in the form of a lemma:

\begin{lem}
$G$ is strongly $T_3$-reconstructible if and only if it satisfies the following two requirements:
1) Every edge in $G$ is necessary and
2) No additional edge can be added without augmenting $T_3(G)$, which holds if and only if we have $ N(v_1) \neq N(v_2)$ for any two non-adjacent vertices $v_1 \neq v_2\in V(G)$.
\end{lem}

To understand what graphs are strongly $T_3$-reconstructible, we need to understand the first requirement above. We will find some sufficient conditions for it by looking at several induced subgraphs which our edge could be in that will guarantee the necessity of said edge. Here is a lemma listing simple scenarios, where we know the edge in question would be necessary, that will be important for our full characterization later on.

\begin{figure}[htbp]
\centering
\includegraphics{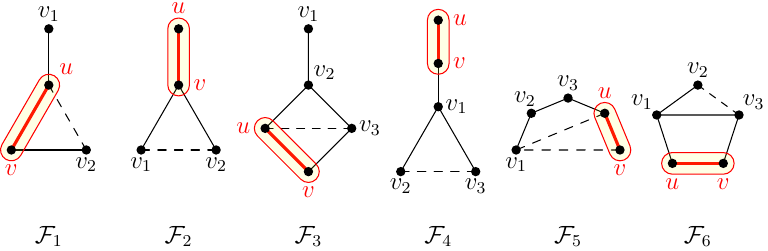}
\caption{The six families that force $uv$ to be necessary. Dotted edges are the ones whose existence we don't care about.}
\label{fig:lemma}
\end{figure}

\begin{lem}
An edge $uv \in E(G)$ is necessary if it is contained in an induced subgraph of $G$ that belongs to one of the families in Figure~\ref{fig:lemma}:
\end{lem}

\begin{proof}
We will show that in each of these 6 cases, deleting $uv$ would inevitably change the set of connected triples of $G$.

$\FF_1$: If we were to delete the edge $uv$, then we would need to add $v_1v$ as an edge for $\{u,v,v_1\}$ to remain a connected triple. But this would create a new connected triple $\{v_1, v, v_2\}$. So we need to delete $vv_2$. However, without the edges $uv$ and $vv_2$, $\{u,v,v_2\}$ will not be a connected triple.

$\FF_2$: If we were to delete the edge $uv$, then we would need to add both $uv_1$ and $uv_2$ as edges to preserve the connected triples $\{u,v,v_1\}$ and $\{u,v,v_2\}$. But this would inevitably create a new connected triple $\{v_1, u, v_2\}$.

$\FF_3$: If we were to delete the edge $uv$, then we would need to add $vv_2$ as an edge to preserve the connected triple $\{u, v, v_2\}$. To prevent creating the connected triple $\{v, v_1, v_2\}$, we would need to delete edge $v_1v_2$. But we know from $\FF_2$ that the edge $v_1v_2$ is a necessary edge here in $\FF_3$.

$\FF_4$: If we were to delete the edge $uv$, then we would need to add $uv_1$ to keep $\{u,v,v_1\}$ as a connected triple. However, since neither $\{u,v_1, v_2\}$ nor $\{u, v_1, v_3\}$ belongs to $T_3(G)$, we would need to delete both $v_1v_2$ and $v_1v_3$, which would destroy $\{v_1, v_2, v_3\}$.

$\FF_5$: If we were to delete the edge $uv$, we would need to add $vv_3$ as an edge for $\{v, u, v_3\} \in T_3(G)$. But since $\{u, v_3, v_2\} \notin T_3(G)$, we would need to delete $v_2v_3$. Yet $\{v_1, v_2, v_3\} \in T_3(G)$ implies that we also need to add $v_1v_3$ as an edge. Now, we have created a connected triple $\{v_1, v_3, v\} \notin T_3(G)$.

$\FF_6$: If we were to delete the edge $uv$, then we would need to add $vv_1$ and $uv_3$ as edges for $\{u,v,v_1\}, \{v, u, v_3\} \in T_3(G)$. But since $\{v,v_1,v_2\}, \{u,v_3,v_2\} \notin T_3(G)$, we would need to delete both $v_1v_2$ and $v_2v_3$. But then $\{v_1, v_2, v_3\} \notin T_3(G)$.

\end{proof}

With Lemma 4.8, we can give a complete characterization for triangle-free graphs. We first prove a very useful lemma.

 \begin{lem}
        Any edge that is not contained in any triangle in a graph $G$ on $n \geq 5$ vertices with the property that $N(v_1) \neq N(v_2)$ for any two non-adjacent vertices $v_1$ and $v_2$ is necessary.
    \end{lem}

    \begin{proof}
         Fix an edge $v_1v_2 \in E(G)$. We will show that $v_1v_2$ is necessary by showing that it is contained in an induced subgraph that belongs to one of the six families of graphs listed in Lemma 4.8. Since $G$ is connected, $v_1v_2$ exists in a connected triple, say $\{v_1, v_2, v_3\}$, with edges $v_1v_2$ and $v_2v_3$. Because $v_1$ and $v_3$ are non-adjacent, we know $N(v_1) \neq N(v_3)$. If there exists a vertex $v \in N(v_1) \setminus N(v_3)$, then the edge $v_1v_2$ exists in an induced subgraph $G[\{v,v_1, v_2, v_3\}]$ that belongs to $\FF_1$.

         Otherwise, there exists a vertex $v \in N(v_3) \setminus N(v_1)$ and $N(v_1) \subsetneq N(v_3)$. If $v_2v \in E(G)$, then $\{v_1, v_2, v_3, v\}$ would induce a graph that belongs to  $\FF_2$. If this is not the case, then $\{v_1, v_2, v_3, v\}$ would induce a $P_4$. Since $G$ is connected and $n \geq 5$, we know there exists $z \notin \{v_1, v_2, v_3, v\}$ that is adjacent to at least one of the vertices in $\{v_1, v_2, v_3, v\}$. If $v_1z \in E(G)$, then $N(v_1) \subsetneq N(v_3)$ would imply that $v_3z \in E(G)$ and that $\{v_1, v_2, v_3, v, z\}$ induces a graph in $\FF_3$. Else if $v_2z \in E(G)$, then $\{v_1, z, v_2, v_3\}$ induces a graph in $\FF_2$. Else if $v_3z \in E(G)$, then $G[\{v_1, z, v_2, v_3\}]$ belongs to $\FF_4$. Else, we know $vz \in E(G)$ and in which case $G[\{v_1, v_2, v_3, v, z\}]$ belongs to $\FF_5$.
    \end{proof}

\begin{thm}
    A triangle-free graph $G$ on $n \geq 5$ vertices is strongly $T_3$-reconstructible if and only if it has the the property that $N(v_1) \neq N(v_2)$ for any two non-adjacent vertices $v_1$ and $v_2$.
\end{thm}
\begin{proof}
    Note the property in question forms a necessary and sufficient condition for meeting the second requirement in Lemma 4.7. So the ``only if" direction is immediate. We just need to show the ``if" direction (i.e. that satisfying this property also guarantees that every edge in G is necessary, which is the first requirement in Lemma 4.7).

    Applying Lemma 4.9 to $G$ gives us that every edge in $G$ is necessary. Thus $G$ is strongly $T_3$-reconstructible by Lemma 4.7.

\end{proof}

\begin{rem}
    Observe that Theorem 4.10, when applied to trees, is equivalent to saying that a tree on $n \geq 5$ vertices is strongly $T_3$-reconstructible if and only if no two leaves share the same parent.
\end{rem}

If we were to derive a complete characterization of all strongly $T_3$-reconstructible graphs, we would need a necessary and sufficient condition for edges that are contained in a triangle to be necessary. It turns out that the six scenarios, which force their highlighted edges to be necessary, respectively, shown in Lemma 4.8 are exactly what we need. Here we state and prove our full characterization of strongly $T_3$-reconstructible graphs.

\begin{figure}[htbp]
\centering
\includegraphics{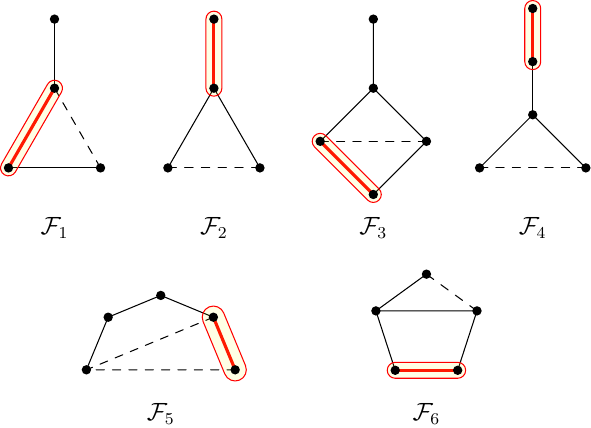}
\caption{The families of graphs mentioned in Theorem 4.12.}
\label{fig:main}
\end{figure}

\begin{figure}[hbtp]
\centering
\includegraphics{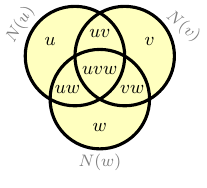}
\caption{The 2-coloring that represents the case where all three vertices of the triangle have at least one private neighbor, all pairs of them have at least one exclusive common neighbor, and there exists at least one vertex that is adjacent to all three vertices.}
\label{fig:venn}
\end{figure}

\begin{thm}
    A graph $G$ on $n \geq 5$ vertices is strongly $T_3$-reconstructible if and only if $N(v_1) \setminus \{v_2 \} \neq N(v_2) \setminus \{v_1 \}$ for all $v_1 \neq v_2 \in V(G)$, and every edge contained in a triangle is also contained in an induced subgraph belonging to one of the families of graphs in Figure~\ref{fig:main}:
\end{thm}

\noindent\emph{Proof of main theorem.}
    The ``if" direction follows directly from Lemma 4.7, Lemma 4.8, and Lemma 4.9. As for the ``only if" direction, it is clear that if there exists $v_1 \neq v_2 \in V(G)$ such that $N(v_1) \setminus \{v_2\} = N(v_2) \setminus \{v_1\}$, then whether $v_1$ and $v_2$ are adjacent would not make a difference to the set of connected triples of $G$. So if $G$ is strongly $T_3$-reconstructible, then it is definitely the case that $N(v_1) \setminus \{v_2 \} \neq N(v_2) \setminus \{v_1 \}$ for all $v_1 \neq v_2 \in V(G)$. What we need to show is that if an edge is contained in a triangle in a graph that satisfies $N(v_1) \setminus \{v_2 \} \neq N(v_2) \setminus \{v_1 \}$ for all $v_1 \neq v_2 \in V(G)$ but is not contained in an induced subraph that belongs to one of the families in Figure~\ref{fig:main}, then it is not necessary.

    To do that, we would take a triangle which our edge is contained in, say triangle $uvw$, and look at how it is connected to the rest of the graph. In particular, we will be interested in how the vertices outside of the triangle but in the neighborhood of $\{u,v,w\}$, which we denote as $N[\{u,v,w\}]$, are connected to the vertices $u, v, w$ respectively. A few words on notation and diction: we will call a vertex outside of the triangle an \emph{exclusive common neighbor} of two vertices of the triangle if said vertex is adjacent to the two vertices but not to the third vertex of the triangle. Since we will only be considering if two vertices have the same set of neighbors outside of themselves, when we write $N(x) \setminus N(y)$ for some adjacent vertices $x$ and $y$, we do not mean to include $y$. Similarly, we will use $N(x) = N(y)$ as short for $N(x)\setminus \{y\} = N(y) \setminus \{x\}$ when $x$ and $y$ are adjacent.

    We can express all such possibilities by all 2-colorings of the Venn diagram, where the regions colored yellow are exactly the ones that are non-empty, shown in Figure~\ref{fig:venn}, up to symmetry. Applying Burnside's lemma gives that there are $40$ scenarios. However, we first note that if there exist at least two vertices in $\{u, v, w\}$ that have private neighbors, say $x$ and $y$, where we say a vertex $u$ in the triangle has a \emph{private} neighbor if there is a vertex outside of the triangle that is adjacent to $u$ but not to $v$ or $w$, then we will find our edge and our triangle $uvw$ in an induced subgraph $G[\{v,u,w,x,y\}]$ that is the bull graph, or the bull graph plus an edge (triangle embedded in a corner of a pentagon), which we know are strongly $T_3$-reconstructible. Furthermore, in either case, our edge would exist in an induced subgraph that belongs to $\FF_1$, which would be a contradiction. On the other hand, the condition that $N(v_1) \setminus \{v_2 \} \neq N(v_2) \setminus \{v_1 \}$ for all $v_1 \neq v_2 \in V(G)$ we imposed on our graph would eliminate the cases where there exist two vertices of the triangle---say $u$ and $v$---that share all the same neighbors outside of the triangle $\left( \text{i.e. } N[\{u,v,w\}] \cap N(u) = N[\{u,v,w\}] \cap N(v) \right)$.

    This leaves us with 12 cases to consider, which we will break into two groups: the group of cases where none of the three vertices have a private neighbor (Cases 1 - 4), and the group of cases where there exists one of $\{u,v,w\}$, say $u$, with a private neighbor (Cases 5 - 12). What we want is to show that for any of the edges in the triangle in each case, if we assume that it is not contained in an induced subgraph that belongs to one of the families in Figure~\ref{fig:main}, then it would not be a necessary edge. \hfill
\break

\mycomment{
    \def\firstcircle{(90:0.7cm) circle (1 cm)}
  \def\secondcircle{(210:0.7cm) circle (1 cm)}
  \def\thirdcircle{(330:0.7cm) circle (1 cm)}
    \begin{tikzpicture}

    \begin{scope}[even odd rule]% first circle without the second
            \clip \secondcircle (-3,-3) rectangle (3,3);
        \fill[yellow] \firstcircle;
    \end{scope}

      \begin{scope}
    \clip \secondcircle;
    \fill[yellow] \thirdcircle;
      \end{scope}

      \begin{scope}
    \clip \firstcircle;
    \fill[yellow] \thirdcircle;
      \end{scope}

      \draw \firstcircle node[text=black,above] {$A$};
      \draw \secondcircle node [text=black,below left] {$B$};
      \draw \thirdcircle node [text=black,below right] {$C$};
    \end{tikzpicture}
}

    \begin{wrapfigure}{l}{3.5cm}
       \includegraphics{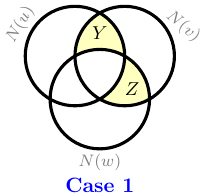}
    \end{wrapfigure}

    \maincase{Case 1} [the only non-empty sets are Y: set of exclusive neighbors of $u$ and $v$; Z: set of exclusive neighbors of $v$ and $w$]

    \begin{proof}
         By symmetry, we need to show $uv$ and $uw$ are not necessary if they are not contained in an induced subgraph that belongs to one of the families of \ref{fig:main}, respectively. Note that $Z = \{z\}$ otherwise $G[\{u,v,z, z_2\}], G[\{u,w,z,z_2\}] \in \FF_2$ for some $z_2 \neq z \in Z$, a contradiction. So the only edge we need to add to preserve $T_3(G)$ if we were to delete edge $uv$ is $uz$. For all $p \in N(z) \setminus N(u)$, if it is an exclusive common neighbor of $v$ and $w$, then $G[\{u,v,z,p\}], G[\{u,w,z,p\}] \in \FF_2$, a contradiction. Else, $p \notin N[\{u,v,w\}]$. But then we would have $G[\{p,z,v,u,w\}] \in \FF_3$, another contradiction. Hence such $p$ does not exist. On the other hand, $z$ has to be adjacent to all $y_i \in Y$, otherwise $G[\{u,y_i, v, z\}] \in \FF_1$ for some $y_i \in Y$, which would be a contradiction. Therefore $N(u) \subset N(z)$ and thus $N(u) = N(z)$. This shows that $uv$ is not necessary for it can be replaced by $uz$. As for $uw$, we would need to replace it with $uz$ and $yw$: in this case there does not exist another $y_2 \neq y \in Y$ as $G[\{y, y_2, u, w\}] \notin \FF_2$. We have shown that adding $uz$ will not cause trouble when our edge in question is $uv$ or $uw$. The argument for $yw$ is entirely symmetric.
    \end{proof}

    \begin{wrapfigure}{l}{3.5cm}
        \includegraphics{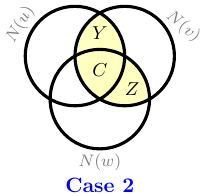}
    \end{wrapfigure}

    \maincase{Case 2}: [Case 1 +  C: the set of common neighbors of $u$, $v$, and $w$]

    \begin{proof}
        We can apply our argument in Case 1. We only need to check that the addition of $uz$ won't create the connected triple $\{u,c_l, z\}$ if $c_lz \notin E(G)$ for some $c_l \in C$. However, if there does exist $c_l \in C$ such that $c_lz \notin E(G)$, then $G[\{c_l, u, v, z\}], G[\{c_l, u, w, z\}] \in \FF_1$, a contradiction. This shows that $uv$ and, by symmetry, $vw$ is not necessary. For $uw$, we also need to check the same problem can't happen with the addition of $wy$. But the argument again is entirely symmetric.
    \end{proof}

    \begin{wrapfigure}{l}{3.5cm}
        \includegraphics{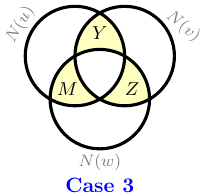}
    \end{wrapfigure}

    \maincase{Case 3}: [the only non-empty sets are Y: exclusive common neighbors of $u$ and $v$; Z: exclusive common neighbors of $v$ and $w$; M: exclusive common neighbors of $u$ and $w$]

    \begin{proof}
        We only need to show $uv$ is not necessary due to symmetry. Again, it's easy to check that both $Z$ and $M$ have only one element, say $z$ and $m$, and that $zy_i, my_i\in E(G)$ for all $y_i \in Y$. Also note $mz \in E(G)$ since $G[\{m,u,v,z\}] \notin \FF_1$. Deleting $uv$ would require the addition of $uz$ and $vm$. If there exists $p \in N(z) \setminus N(u)$. Then $p$ is either an exclusive common neighbor of $v$ and $w$ or $p \notin N[\{u,v,w\}]$. The former can't happen since $z$ is the unique exclusive common neighbor of $v$ and $w$ and if it were the latter then $G[\{u,w,z,p,v\}] \in \FF_3$, a contradiction. Hence $N(z) \subset N(u)$. On the other hand, having $zy_i, zm, zv, zw \in E(G)$ for all $y_i \in Y$ implies that $N(u) \subset N(z)$ and thus $N(u) = N(z)$. This shows that adding $uz$ as an edge does not create any new connected triples. Similarly, we can show that the same is true for $vm$.
    \end{proof}

    \begin{wrapfigure}{l}{3.5cm}
        \includegraphics{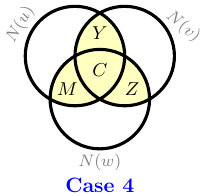}
    \end{wrapfigure}

    \maincase{Case 4}: [Case 3 + C: the set of common neighbors of $u$, $v$, and $w$]

    \begin{proof}
        Similar to Case 3, we will only need to show it for $uv$ (i.e. $uv$ is not necessary if it is not contained in an induced subgraph belonging to one of the families in \ref{fig:main}). As we have shown in Case 3, $u$ and $w$ have an unique exclusive common neighbor, say $m$, and $v$ and $w$ have an unique exclusive common neighbor, say $z$ and $mz \in E(G)$. It is sufficient to add $uz$ and $vm$ to make up for the connected triples destroyed from the deletion of $uv$. We will show that adding $uz$ does not create any new connected triples that were not in $T_3(G)$ by showing that $u$ and $z$ have the same set of neighbors. The argument for $vm$ is symmetric. If there exists $p \in N(z) \setminus N(u)$, then it has to be the case that $p \notin N[\{u,v,w\}]$, which would imply that $G[\{u,v,z,p,w\}] \in \FF_3$, a contradiction. On the other hand, recall $mz \in E(G)$ and that $zy_i, zc_l \in E(G)$ for all $y_i \in Y$ and $c_l \in C$, otherwise $G[\{u,v,y_i,z\}], G[\{u,v,c_l,x\}] \in \FF_1$ for some $c_l \in C$ or $y_i \in Y$, which would be a contradiction. Therefore $N(u) \subset N(z)$ and thus $N(u) = N(z)$.
    \end{proof}

    For all the cases where $u$ has a private neighbor, say $x_k$, with regards to the triangle $uvw$, both $uv$ and $uw$ are contained in the induced subgraph $G[\{x_k, u, w, v\}]$ in such a way that belongs to $\FF_1$, respectively. So the only edge in this triangle that might not be contained in an induced subgraph belonging to any of the families of \ref{fig:main} is edge $vw$. Hence we only check when our edge in question is $vw$ in Cases 5 to 12. So when we say some induced subgraph that $vw$ is in belongs to some family in \ref{fig:main}, it will always be implicitly assumed that it will be with regards to $vw$, where the $vw$ will correspond to the highlighted edge in the said family in \ref{fig:main}. Recall we are assuming that $vw$ is not contained in any induced subgraph that belongs to some family in  \ref{fig:main} and we want to show that $vw$ is not necessary. \hfill
    \break

    \begin{wrapfigure}{l}{3.5cm}
       \includegraphics{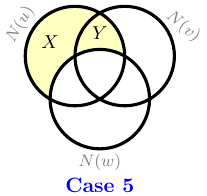}
        \par\vspace{1ex}
        \includegraphics{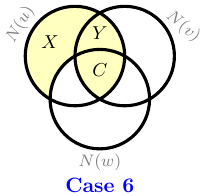}
    \end{wrapfigure}

    \maincase{Case 5}: [the only non-empty sets are X: set of private neighbors of $u$; Y: set of exclusive common neighbors of $u$ and $v$;]

    \maincase{Case 6}: [Case 5 +  C: the set of common neighbors of $u$, $v$, and $w$]
    \begin{proof}
         If there exists $y_i \neq y_j \in X$, then $G[\{y_i, v, y_j, w\}] \in \FF_2$, a contradiction. So $u$ and $v$ have an unique exclusive common neighbor $y$. Also, there exists an unique $x \in X$ such that $xy \in E(G)$. If we were to delete edge $vw$, the only connected triple destroyed would be $\{w,v,y\}$. So we would need to add $wy$ as an edge. If there exists $c_l \in C$ such that $c_ly \notin E(G)$, then $G[\{c_l,w,v,y\}] \in \FF_1$, a contradiction. Hence $N(w) \subset N(y)$. For any $p \in N(y) \setminus N(x)$, $p$ is either a private neighbor of $u$ or not in $N[\{u,v,w\}]$ at all. However, the former gives $G[\{u,w,v,y,p\}] \in \FF_3$, a contradiction. So $p$ would have to be the unique private neighbor of $u$ that is adjacent to $y$. This means that $\{w,y,p\}$ is the unique new connected triple created by the addition of $wy$ as an edge. Thus we delete the edge $yp$.

         We will show that $\{v,y,p\}$ is the only connected triple in $T_3(G)$ destroyed by deleting $yp$ and thus we would only need to add $vp$ as an edge to preserve $T_3(G)$. Suppose there exists $p_2 \neq v \in N(y) \setminus N(p)$. If $p_2 \notin N[\{u,v,w\}]$, then we have $G[\{p_2, y, u, w, v\}] \in \FF_3$, a contradiction. Else if $p_2$ is a private neighbor of $u$, then $G[\{p_2, p, y, v, w\}] \in \FF_2$, a contradiction. Else $p_2$ has to be a common neighbor of all $u$, $v$, and $w$, which would give $G[\{p, y, v, w, p_2\}] \in \FF_3$, another contradiction. Thus no such $p_2$ exists and $N(y) \setminus N(p) = \{v\}$. Now suppose there exists $q \neq y \in N(p) \setminus N(y)$. If $q \notin N[\{u,v,w\}]$ or $q$ is a private neighbor of $u$, then $G[\{q, p, y, v, w\}] \in \FF_5$, a contradiction. Otherwise, $q$ is a common neighbor of all three vertices, which gives $G[\{v, w, q, y\}] \in \FF_1$, which is a contradiction.

         Lastly, we will show that adding $vp$ does not create any more connected triples. Since $y$ is the unique exclusive common neighbor of $u$ and $v$, any $q \in N(v) \setminus N(p)$ would have to be a common neighbor of $u$, $v$, and $w$. But this would give $G[\{p,w,v,q\}] \in \FF_1$, a contradiction. On the other hand, suppose there exists $q \in N(p) \setminus N(v)$. We showed earlier that there does not exist any $q \neq y$ satisfying $qp \in E(G)$ but $qy \notin E(G)$. Thus this $q \in N(p) \setminus N(v)$ is also adjacent to $y$. If $q \notin N[\{u,v,w\}]$, then $G[\{q, y, u, w, v\}] \in \FF_3$, a contradiction. Otherwise, $q$ is a private neighbor of $u$, in which case we have $G[\{w,v,y,p,q\}] \in \FF_4$, which is a contradiction. Therefore, we have shown that $vw$ is not necessary in this case because deleting $vw$ and $yp$ while adding $yw$ and $vp$ preserves the set of connected triples.
    \end{proof}\hfill
    \break

    \begin{wrapfigure}{l}{3.5cm}
       \includegraphics{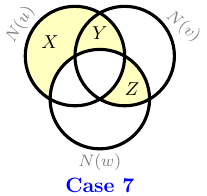}
        \par\vspace{1ex}
        \includegraphics{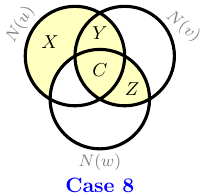}
    \end{wrapfigure}

    \maincase{Case 7} [the only non-empty sets are X: the set of private neighbors of $u$; Y: the set of exclusive common neighbors of $v$ and $u$; Z: the set of exclusive common neighbors of $v$ and $w$.]

    \maincase{Case 8} [Case 7 +  C: the set of common neighbors of $u$, $v$, and $w$].
    \begin{proof}
        First, note that $Y$ has only one element $y$, otherwise $G[\{y, y_i, v, w\}]$ belongs to $\FF_2$ (in Figure~\ref{fig:main}) for some $y_i \neq y \in Y$. If there exists $z_j \in Z$ such that $yz_j \notin E$, then $G[\{y, v, z_j, w\}]$ would be in $\FF_2$, a contradiction. So we assume that $yz_j \in E(G)$ for all $z_j$. Similarly, we have $yc_l \in E(G)$ for all $c_l \in C$.

        If we were to delete edge $vw$, we would need to add $wy$ as an edge to keep $\{v, w, y\}$ as a connected triple. Note that $N(w) \setminus N(v) = \emptyset$ and $N(v) \setminus N(w) = \{y\}$. This shows that deleting $vw$ does not destroy any more connected triples. Now we turn our attention to the addition of $wy$.

        Since all $c_l$ and $z_j$, as well as $u$ and $v$, are adjacent to $y$, we have that $N(w) \setminus N(y) = \emptyset$. Suppose there exists $p \in N(y) \setminus N(w)$. If $p$ is not in $N[\{u,v,w\}]$, then $G[\{p, y, u, v, w\}]$ would be in $\FF_3$, a contradiction. The fact that $Y$ has only one element $y$ and that $p \notin N(w)$ implies that $p$ has to be a private neighbor of $u$. If there exists $p_2 \neq p$ such that $p_2$ is also a private neighbor of $u$ and is adjacent to $y$, then $\{p, p_2, y, v, w\}$ would induce a subgraph that belongs to $\FF_4$, a contradiction. Therefore such $p$ is unique (i.e. $N(y) \setminus N(w) = \{p\}$).

        So we need to delete $py$. Observe that there does not exist $\Tilde{p} \in N(p) \setminus N(y)$ because $\{\Tilde{p}, p, y_i, v, w\}$ cannot induce a subgraph in $\FF_5$. Suppose there exists $\Tilde{p} \neq v \in N(y) \setminus N(p)$. If there exists $z_j \in Z$ or $c_l \in C$ such that $pz_j \notin E(G)$ or $pc_l \notin E(G)$, then said $z_j$ or $c_l$, together with $\{p, y, v, w\}$, would induce $\FF_3$, which is a contradiction. So $pz_j, pc_l \in E$ for all $z_j \in Z$ and all $c_l \in C$. This means that all common neighbours of $v$ and $w$ are contained in $N(p)$ and thus $\Tilde{p}$ is not a common neighbor of $v$ and $w$. Furthermore, $y$ is the unique exclusive common neighbor of $v$ and $u$. This implies that $\Tilde{p}$ is a private neighbor of $u$. However, then $\{ p, \Tilde{p}, y, v, w\}$ would induce a subgraph that belongs in $\FF_4,$ a contradiction. Hence the only vertex in $N(y)\setminus N(p)$ is $v$. To preserve the original connected triple $\{p, y, v\}$, we need to add $pv$ as an edge.

        Up until the addition of $pv$ as an edge, we have done everything needed to preserve the set of connected triples. Now we want to show that adding $pv$ doesn't create any new connected triples and thus no more changes are needed. We have shown previously that $N(p) \subset N(y)$. So if there exists any $p\prime \in N(p)$ that is not adjacent to $v$, which implies that it is also not adjacent to $w$, then $G[\{p, p\prime, y, v, w\}]$ would belong to $\FF_4$, a contradiction. Therefore everything in $N(p)$ must be a neighbor of $v$ and thus $N(p) \setminus N(v) = \emptyset$. On the other hand, $N(v) \setminus N(p)$ is empty as well since every neighbor of $v$ is either $y$ or a common neighbor of $v$ and $w$, all of which we've shown earlier are adjacent to $p$. Since $p$ and $v$ share all the same neighbors, adding $pv$ won't add any new connected triples.

    \end{proof}

    \begin{wrapfigure}{l}{3.5cm}
       \includegraphics{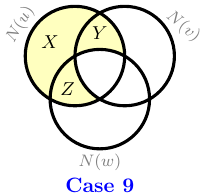}
        \par\vspace{1ex}
        \includegraphics{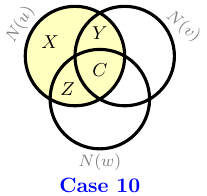}
    \end{wrapfigure}

    \maincase{Case 9}  [the only non-empty sets are
     X: the set of private neighbors of $u$; Y: the set of exclusive neighbors of $u$ and $v$; Z: the set of exclusive neighbors of $u$ and $w$]

   \maincase{Case 10} [Case 9 + C: the set of common neighbors of $u$, $v$, and $w$ not being empty]

    \begin{proof}
         Note that if there exists $z_i \neq z_j \in Z$, then $\{v, w, z_j, z_i\}$ would induce a graph that belongs to $\FF_3$, a contradiction. Therefore, we have $Z = \{z\}$ and, very similarly, $Y = \{y\}$. Assume $z$ and $y$ are adjacent, otherwise $vw$ will be contained in an induced subgraph $G[\{z, w, v, y\}]$ that belongs to $\FF_1$. Furthermore, assume $yc_l, zc_l \in E(G)$ for all $c_l \in C$, otherwise $\{y, v, w, c_{l_1}\}$ or $\{z, w, c_{l_2}, v\}$ would each induce a subgraph that belongs to $\FF_1$ for some $c_{l_1}, c_{l_2} \in C$.

         If we were to delete $vw$, the only edges we need to add are $vz$ and $wy$. We will show that none of these additions will create new connected triples by showing that the two vertices of any of the two edges have the same set of neighbors outside of themselves. By symmetry, it would be enough to show this for $vz$.

        Since $N(v) = Y \cup \{w\} = \{w,y\}$ and $zy, zw \in E(G)$, we have that $N(v) \setminus N(z) = \emptyset$. On the other hand, suppose there exists $p \in N(z) \setminus N(v)$. If $p$ is a private neighbor of $u$ with regards to the triangle $uvw$ or if $p \notin N[\{u,v,w\}]$, then $G[\{w,v,y,z,p\}]$ belongs to $\FF_6$, a contradiction. Then $p$ must be a common neighbor of $u$ and $w$ that is not adjacent to $v$. (i.e. $p \in Z$). But this would mean that $G[\{z, p, w, v\}]$ belongs to $\FF_2$, a contradiction. Thus $N(z) \setminus N(v) = \emptyset$. Therefore $N(z) = N(v)$.

        This shows that $vw$ is not necessary as replacing it with $wy$ and $vz$ does not change the set of connected triples. \hfill
        \break
    \end{proof}

    \begin{wrapfigure}{L}{3.5cm}
       \includegraphics{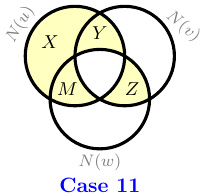}
        \par\vspace{1ex}
        \includegraphics{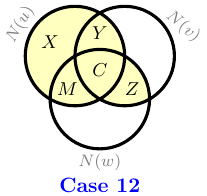}
    \end{wrapfigure}

    \maincase{Case 11} [the only non-empty sets are X: the set of private neighbors of $u$; Y: the set of exclusive neighbors of $U$ and $v$; Z: the set of exclusive neighbors of $v$ and $w$; M: the set of exclusive neighbors of $u$ and $w$]

    \maincase{Case 12} [Case 11 + C: the set of common neighbors of $u$, $v$, and $w$ not being empty]
        \begin{proof}
        If there exists $y_i \neq y_j \in Y $, then $\{y_i, y_j, v, w\}$ will induce a subgraph that belongs to $\FF_2$, a contradiction. So $Y = \{y\}$ and similarly $M = \{m\}$. Also assume $ym \in E(G)$ otherwise $G[\{m,w, v, y\}]$ belongs to $\FF_1$, which is a contradiction.

        Since $N(v) \setminus N(w) = \{y\}$ and $N(w) \setminus N(v) = \{m\}$, the only connected triples destroyed by deleting edge $vw$ are $\{w,v,y\}$ and $\{v,w,m\}$. So we need to add $wy$ and $mv$ as edges. We would like to show such additions do not create new connected triples by showing that the two vertices of either one of the two edges have the same set of neighbors outside of themselves. Again, it would be sufficient to check this for one of them, say $wy$, due to symmetry.

        Recall $ym \in E(G)$ and note that $G[\{y,v, w, z_j\}], G[\{y,v,w,c_l\}] \notin \FF_1$ implies that $yz_j, yc_l \in E(G)$ for all $z_j \in Z$ and $c_l \in C$. This shows that $N(w) \setminus N(y) = \emptyset$. On the other hand, suppose there exists $p \in N(y) \setminus N(w)$. Then either $p \notin N[\{u,v,w\}]$ or it is a private neighbor of $u$ (i.e. $p \in X$). However, the former would imply that $G[\{p,y,u,v,w\}] \in \FF_3$, a contradiction, whereas the latter $G[\{p,m,w,v,y\}] \in \FF_6$, which is also a contradiction. Therefore $N(w) = N(y)$.

        \end{proof}

This finishes the proof of the main theorem. $\blacksquare$

\section{Future Directions}
Given our complete characterization of strongly $T_3$-reconstructible graphs in Theorem 4.12, one might wonder whether the collection of graphs contained in the families listed in the statement is minimal; or what other equivalent characterizations there might exist. Naturally, one could also ask about characterizations of strongly $T_k$-reconstructible graphs for $k \geq 4$, which we can analogously define to be graphs that could be uniquely reconstructed from $T_k$ without any additional information. Furthermore, one could also study the necessary and sufficent conditions on a class of graphs for it to be $T_3$-reconstructible and try to extend it to an arbitrary $k$. We leave these questions for future work.

\section*{Acknowledgements}
This research was conducted at the 2023 University of Minnesota Duluth REU, funded by Jane Street Capital and the National Security Agency. The author would like to thank Noah Kravitz for suggesting the problem, Evan Chen for helping make the figures in the paper, and Yelena Mandelshtam and Andrew Kwon for providing invaluable feedback on the presentation of the paper. Last but certainly not least, the author is deeply grateful to Joe Gallian and Colin Defant for organizing the Duluth REU.
\hfill
\break

\end{document}